\documentclass[10pt,reqno]{amsart}

\usepackage{color}
\usepackage{enumerate}

\newtheorem{thm}{Theorem}[section]
\newtheorem{defn}[thm]{Definition}
\newtheorem{corollary}[thm]{Corollary}
\newtheorem{lemma}[thm]{Lemma}

\theoremstyle{remark}
\newtheorem{remark}[thm]{Remark}

\def\qed{{\hfill $\Box$ \bigskip}}

\def\XXint#1#2#3{{\setbox0=\hbox{$#1{#2#3}{\int}$}
\vcenter{\hbox{$#2#3$}}\kern-.5\wd0}}

\newcommand\cbrk{\text{$]$\kern-.15em$]$}}
\newcommand\opar{\text{\,\raise.2ex\hbox{${\scriptstyle
|}$}\kern-.34em$($}}
\newcommand\cpar{\text{$)$\kern-.34em\raise.2ex\hbox{${\scriptstyle |}$}}\,}

\def\<{\langle}
\def\>{\rangle}

\newcommand\bR{\mathbb{R}}

\newcommand\bZ{\mathbb{Z}}

\newcommand\bN{\mathbb{N}}

\newcommand\fR{\mathbf{R}}

\newcommand\cF{\mathcal{F}}
\newcommand\cG{\mathcal{G}}

\newcommand\cI{\mathcal{I}}

\newcommand\cL{\mathcal{L}}

\newcommand\cM{\mathcal{M}}
\newcommand\cT{\mathcal{T}}

\newcommand{\mysection}[1]{\section{#1}
\setcounter{equation}{0}}

\begin{document}

\title[Parabolic pseudu-differential equations in $L_p$-Lipschitz space]
{An $L_p$-Lipschitz  theory  for parabolic equations with time measurable pseudo-differential operators}

\author{Ildoo Kim}
\address{Center for Mathematical Challenges, Korea Institute for Advanced Study, 85 Hoegiro Dongdaemun-gu,
Seoul 130-722, Republic of Korea} \email{waldoo@kias.re.kr}
\thanks{The first author was supported by the TJ Park Science
         Fellowship of POSCO TJ Park Foundation}



\begin{abstract}
In this article we prove the existence and uniqueness of a (weak) solution $u$ in 
$L_p\left( (0,T) ; \Lambda_{\gamma+m}\right)$ to the Cauchy problem
\begin{align}
						\notag
&\frac{\partial u}{\partial t}(t,x)=\psi(t,i\nabla)u(t,x)+f(t,x),\quad (t,x) \in (0,T) \times \mathbf{R}^d \\
						\label{main eqn}
& u(0,x)=0,
\end{align}
where $d \in \mathbb{N}$, $p \in (1,\infty]$, $\gamma,m \in (0,\infty)$, $\Lambda_{\gamma+m}$ is the Lipschitz  space on $\mathbf{R}^d$ whose order is $\gamma+m$, $f \in L_p\left( (0,T) ; \Lambda_{\gamma} \right)$, and $\psi(t,i\nabla)$ is a time measurable pseudo-differential operator whose symbol is $\psi(t,\xi)$, i.e.
$$
\psi(t,i\nabla)u(t,x)=\cF^{-1}\left[\psi(t,\xi)\cF\left[u(t,\cdot)\right](\xi)\right](x),
$$
with the assumptions
\begin{align*}
\Re[\psi(t,\xi)] \leq  -\nu|\xi|^{\gamma},
\end{align*}
and
\begin{align*}
|D_{\xi}^{\alpha}\psi(t,\xi)|\leq\nu^{-1}|\xi|^{\gamma-|\alpha|}.
\end{align*}
Furthermore, we show 
\begin{align}
					\label{e 1028 1}
\int_0^T \|u(t,\cdot)\|^p_{\Lambda_{\gamma+m}} dt \leq N \int_0^T \|f(t,\cdot)\|^p_{\Lambda_{m}} dt,
\end{align}
where $N$ is a positive constant depending only on $d$, $p$, $\gamma$, $\nu$, $m$, and $T$,

The unique solvability of equation (\ref{main eqn}) in $L_p$-H\"older space is also considered. 
More precisely, for any $f \in L_p((0,T);C^{n+\alpha})$, there exists a unique solution $u \in L_p((0,T);C^{\gamma+n+\alpha}(\mathbf{R}^d))$ to equation (\ref{main eqn}) and for this solution $u$,
\begin{align}
						\label{e 1029 1}
\int_0^T \|u(t,\cdot)\|^p_{C^{\gamma+n+\alpha}}dt \leq N \int_0^T \|f(t,\cdot)\|^p_{C^{n+\alpha}}dt,
\end{align}
where $n \in \mathbb{Z}_+$, $\alpha \in (0,1)$, and $\gamma+\alpha \notin \mathbb{Z}_+$.
\end{abstract}

\maketitle

\mysection{Introduction}

The class of pseudo-differential operators is a very large class of differential operators including 
second-order, $2m$-order,  and generators of Markov processes.
Therefore, theories for pseduo-differential operators have been applied in many areas of science
and contain lots of mathematical interesting properties. 
For classical and modern theories to pseudo-differential operators, we refer books \cite{hormander2007analysis, Stein1993, Krylov2008, jacob2002pseudo, abels2012pseudodifferential}.

Pseudo-differential operators are treated mostly in elliptic setting
and commonly independent of $t$ or regular with respect to $t$ even though there are a few results in parabolic setting.

Recently, the author with collaborators  studied pseudo-differential operators which have no regularity with respect $t$.
In \cite{Kim2014BMOpseudo, kim2016q}
we obtained  BMO estimates and $L_q((0,\infty); L_p)$ estimates for the singular integral operator 
$$
\cT f(t,x):= \int_0^t \int_{\fR^d} K(t,s,x-y)f(s,y)dyds,
$$
where
\begin{align*}
K(s,t,x) := \cF^{-1} \left[ |\xi|^\gamma \exp\left( \int_s^t \psi(r,\xi)dr \right) \right]
\end{align*}
and the symbol $\psi(t,\xi)$ satisfies (\ref{sym1}) and (\ref{sym2}).

In this artilce, we study the well-posedness of Cauchy problem (\ref{main eqn}) in the $L_p$-Lipschitz space $L_p\left( (0,T) ; \Lambda_{\gamma+m}\right)$ (see Definition \ref{d 1029 1} for the Lipschitz space $\Lambda_{\gamma + m}$).
Especially, we obtain the optimal regularity estimate (\ref{e 1028 1}) when the given datum $f$ is in $L_p\left( (0,T) ; \Lambda_{m}\right)$. To the best of our knowledge,  this article is the first result which handles the unique solvability of equations with pseudo-differential operators in the Lipschitz  space, although there are a few results related to the boundedness of certain elliptic pseudo-differential operators in the Lipschitz space  (see \cite[Chapter VI]{Stein1993} and \cite{lin2010pseudo}).


There is a close relation between the Lipschitz space and the classical H\"older space.
If $\gamma \in (0,1)$, then $\Lambda_{\gamma}=C^\gamma$.
In general, Lipschitz space $\Lambda_{\gamma}$ is a bigger class than the classical H\"older space $C^{\gamma}$. 
Hence it is needed to remark related resluts in $L_p$-H\"older space.   If the operator is second-order, then there exists an $L_p$-H\"older theory. 
In \cite{Krylov2002}, Krylov obtained the unique solvability to the  parabolic second-order equation
$$
u_t(t,x)= a^{ij}(t)u_{x^ix^j}(t,x) - \lambda u(t,x) \qquad (t,x) \in \fR^{d+1}
$$
in $L_p(\fR ; C^{2+\alpha})$-space, where $p \in (1,\infty]$ and $\alpha\in (0,1)$.
Here $a^{ij}(t)$ are merely measurable  and satisfy an ellipticity condition, i.e. there exists a constant $\delta >0$ so that
$$
\delta \leq a^{ij}(t) \xi^{i} \xi^j \leq \delta^{-1} \qquad \forall (t,\xi) \in \fR^{d+1}.
$$ 
Except Krylov's work, we could not find any other result studying parabolic equations in $L_p$-H\"older space with all $p \in (1,\infty]$.
However, if we restrict $p = \infty$, many results can be found.
We refer the reader to \cite{lorenzi2000optimal,dong2011partial,dong2015partial}(second-order equations) and  
\cite{Mikulevivcius1992,mikulevicius2014cauchy}(integro-differential equations).

The novelty of our result is that we handle parabolic equations with arbitrary  positive order operator and the unique solvability is considered in the space $L_p\left( (0,T) ; \Lambda_{\gamma+m}\right)$ which is rougher than $L_p\left( (0,T) ; C^{\gamma+m}\right)$ with all $\gamma, m >0$.
We emphasize that our estimates hold for all $p \in (1,\infty]$.
In this sense, even the $L_p$-H\"older estimate (\ref{e 1029 1}) given by an application of (\ref{e 1028 1}) is new since most previous results are proved only when $\gamma \in (0,2]$, $p=\infty$, and $m \in (0,1)$.

Another innovation of this paper is the method we use.  
Maximum principles play an important role in the proofs of \cite{lorenzi2000optimal,dong2011partial,dong2015partial} 
and the proof of \cite{Krylov2002} highly depends on explicit form of  the heat kernel and uppper bounds of its derivatives.
For integro-differential operators, 
the methods obtaining estimates are connected with a probability theory.
For instance, probability tools such as It\^o's formula and the ingetral representation of generators of L\'evy processes are used in the proofs of \cite{Mikulevivcius1992,mikulevicius2014cauchy}.
However we do not have such  rich information for pseudo-differential operators
and thus the method we use in this article is different from previous one. 
We adopt Littlewood-Paley operators which are recognized as one of most powerful tools in modern Fourier analysis. 
Since most computations are related to the Fourier transforms of kernels instead of kernels themselves,
calculation becomes much simpler even though estimates in this paper are stronger than previous results.

This article is organized as follows. We introduce our main results in Section 2.
In section 3, we prove required kernel estimates related to pseudo-differential operators.
In section 4, an $L_\infty \left( (0,T) ; \Lambda_{\gamma+m}\right)$-estimate is obtained.
Finally, proofs of main theorems are given in Section 5.

We finish the introduction with  notation used in the article.
\begin{itemize}
\item $\bN$ and $\bZ$ denote the natural number system
and the integer number system, respectively.
$\bZ_+$ is the subset of $\bZ$ whose elements are nonnegative, i.e. $\bZ_+ :=\{ k \in \bZ ; k  \geq 0\}$. 
As usual $\fR^{d}$
stands for the Euclidean space of points $x=(x^{1},...,x^{d})$.
 For $i=1,...,d$, multi-index $\alpha=(\alpha_{1},...,\alpha_{d})$,
$\alpha_{i}\in\{0,1,2,...\}$, and a function $u(x)$ we set
\begin{align*}
u_{x^{i}}=\frac{\partial u}{\partial x^{i}}=D_{i}u,\quad
D^{\alpha}u=D_{1}^{\alpha_{1}}\cdot...\cdot D^{\alpha_{d}}_{d}u,
\quad  \nabla u=(u_{x^1}, u_{x^2}, \cdots, u_{x^d}).
\end{align*}
Sometimes we use $D^\alpha_x$ to denote the variable to which differentiation is taken. 
$C(\fR^d)$ denotes the space of bounded continuous functions on $\fR^d$.
For $n \in \bN$, we write $u \in C^n(\fR^d)$  if $u$ is $n$-times continuously differentiable in $\fR^d$ and 
the supremum of all derivatives up to $n$ is bounded, i.e. $\sup_{x \in \bR^d, |\alpha|\leq n} |D^\alpha u(x)|< \infty$. 
Simply we put $C^n := C^n(\fR^d)$.

\item For $p \in [1,\infty)$, a normed space $F$,
and a  measure space $(X,\mathcal{M},\mu)$, 
$$
L_{p}(X,\cM,\mu;F)
$$
denotes the space of all $F$-valued $\mathcal{M}^{\mu}$-measurable functions
$u$ so that
\[
\left\Vert u\right\Vert _{L_{p}(X,\cM,\mu;F)}:=\left(\int_{X}\left\Vert u(x)\right\Vert _{F}^{p}\mu(dx)\right)^{1/p}<\infty,
\]
where $\mathcal{M}^{\mu}$ denotes the completion of $\cM$ with respect to the measure $\mu$.

For $p=\infty$, we write $u \in L_{\infty}(X,\cM,\mu;F)$ iff
\begin{align*}
\sup_{x}|u(x)| := \|u\|_{L_{\infty}(X,\cM,\mu;F)}
:= \inf\left\{ \nu \geq 0 : \mu( \{ x: \|u(x)\|_F > \nu\})=0\right\} <\infty.
\end{align*}
If there is no confusion for the given measure and $\sigma$-algebra, we usually omit the measure and $\sigma$-algebra.
In particular, we denote $L_p = L_p(\fR^d,\cL, \ell ;\fR)$,
where $\cL$ is the Lebesgue measurable sets,  and $\ell$ is the Lebesgue measure.

\item We use the notation $N$ to denote a generic constant which may change from line to line. 
If we write $N=N(a,b,\cdots)$, this means that the
constant $N$ depends only on $a,b,\cdots$.

\item We use  ``$:=$" or ``$=:$" to denote a definition. 
For $a,b \in \fR$, $a \wedge b := \min \{ a,b\}$, $a \vee b := \max \{a , b\}$, and
$\lfloor a \rfloor$ is the biggest integer which is less than or equal to $a$.
For a  set $A$, we use $1_A(x)$ to denote  the indicator of $A$, i.e. $1_A(x) = 1$ if $x \in A$ and $1_A(x) =0$ if $x \notin A$. 
For a Lebesgue measurable set $B$, $|B|$ is the Lebesgue measure of $B$. 
For a complex number $z$, $\Re[z]$ is the real part of $z$
and $\bar z$ is the complex conjugate of $z$.

\item By $\cF$ and $\cF^{-1}$ we denote the d-dimensional Fourier transform and the inverse Fourier transform, respectively. That is,
$\cF[f](\xi) := \int_{\fR^{d}} e^{-i x \cdot \xi} f(x) dx$ and $\cF^{-1}[f](x) := \frac{1}{(2\pi)^d}\int_{\fR^{d}} e^{ i\xi \cdot x} f(\xi) d\xi$.
\end{itemize}

\mysection{Main results}
Recall the assumptions on the symbol $\psi(t,\xi)$.
Let $\gamma \in (0,\infty)$ and $\psi(t,\xi)$ be a measurable function on $ [0,T] \times \fR^{d}$ satisfying
\begin{align}
					\label{sym1}
\Re[\psi(t,\xi)] \leq  -\nu|\xi|^{\gamma},\quad \quad \forall \, (t,\xi) \in [0,T]\times \fR^d
\end{align}
and
\begin{align}
						\label{sym2}
|D_{\xi}^{\alpha}\psi(t,\xi)|\leq\nu^{-1}|\xi|^{\gamma-|\alpha|},\quad \quad \forall \, (t,\xi)\in [0,T]\times (\mathbf{R}^{d}\setminus\{0\}), \,\, |\alpha|\leq \left\lfloor\frac{d}{2}\right\rfloor+1,
\end{align}
where $\nu$ is a positive constant. 
The $\left\lfloor\frac{d}{2}\right\rfloor+1$ differentiability on the symbol has been known as an optimal differentiability  $($cf. Mihlin's condition  and H\"ormander's condition \cite[Theorem 5.2.7]{grafakos2008classical}$)$.
For notational convenience, we put 
$$
d_0:=\left\lfloor\frac{d}{2}\right\rfloor+1.
$$
Define a pseudo-differential operator $\psi(t, i\nabla)$ on $C_c^\infty(\fR^d)$ as
\begin{align*}
\psi(t, i\nabla)\phi(x):=\cF^{-1}\left[ \psi(t,\xi) \cF(\phi)(\xi)\right](x)
\end{align*}
and its adjoint operator $\psi^\ast (t, i\nabla)$ as
\begin{align*}
 \psi^\ast (t, i\nabla)\phi(x)
:= \overline{\cF^{-1}\left[ \bar \psi(t,\xi) \cF[\phi](\xi)\right](x)}
= \cF^{-1} \left[ \psi(t,-\xi) \cF[\phi](\xi) \right](x),
\end{align*}
where $\phi \in C_c^\infty(\fR^d)$ and $\bar \psi(t,\xi)$ denotes the complex conjugate of $\psi(t,\xi)$. 
Observe that
\begin{align}
							\label{ad fr}
\cF[ \psi^\ast (t, i\nabla)\phi(x)](\xi)
= \psi(t,-\xi)\cF(\phi)(\xi).
\end{align}
\begin{remark}
If $\psi(t,\xi)$ is defined on a interval $[0,T] \times \fR^d$, then there exists a trivial extension to $\fR^{d+1}$ by putting
$$
\psi(t,\xi) := \psi(0,\xi) \quad t \in (-\infty, 0)
$$ 
and 
$$
\psi(t,\xi) := \psi(T,\xi) \quad t \in (T,\infty).
$$
Therefore we may assume that the symbol $\psi(t,\xi)$ is defined on $\fR^{d+1}$. 
\end{remark}

We adopt the definition of a solution to equation (\ref{main eqn}) in the weak sense as usual.
\begin{defn}[Definition of a solution $u$]
					\label{sol}
Let $T \in (0,\infty)$. We say that a locally integrable function $u$ on $(0,T) \times \fR^d$ is a (weak) solution to equation (\ref{main eqn}) iff
for any $\phi \in C^\infty_c\left((0,T) \times \fR^d\right)$
\begin{align*}
\int_0^T \int_{\fR^d} u(t,x) \left( -\phi_t(t,x) - \psi^\ast (t,i\nabla)\phi(t,x)\right)dtdx
=\int_0^T \int_{\fR^d} f(t,x) \phi(t,x) dtdx
\end{align*}
\end{defn}

We introduce function spaces needed to handle solvability of equation (\ref{main eqn}) in $L_p$-Lipschitz space.
\begin{defn}
For $f \in C(\fR^d)$ and $x,h \in \fR^d$, define the difference operator as
\begin{align*}
D_h(f)(x) := f(x+h)-f(x).
\end{align*}
Inductively, for any $n \in \bN \setminus \{1\}$, we define
\begin{align*}
D_h^n(f)(x) = D_h(D^{n-1}_hf)(x).
\end{align*}
\end{defn}
						\label{d 1029 1}
\begin{defn}[Lipschitz  space]
For $m \in (0,\infty)$ and $f \in C(\fR^d)$, we define 
\begin{align}
							\label{ho lip norm}
\|f\|_{\dot \Lambda_m}:= \sup_{x \in \fR^d}\sup_{h \in \fR^d \setminus \{0\}} \frac{\left|D_h^{\lfloor m\rfloor+1}(f)(x)\right|}{|h|^m}
\end{align}
and
\begin{align}
							\label{lip norm}
\|f\|_{\Lambda_m}:=\|f\|_{L_\infty} + \sup_{x \in \fR^d}\sup_{h \in \fR^d \setminus \{0\}} \frac{\left|D_h^{\lfloor m\rfloor+1}(f)(x)\right|}{|h|^m}.
\end{align}
The space of continuous functions $f$ with $\|f\|_{\dot \Lambda_m}< \infty$ is called homogeneous Lipschitz  space  whose order is $m$,
which is denoted by $\dot  \Lambda_m$.
Similarly, $\Lambda_m$ denotes the spaces of continuous functions $f$ with $\|f\|_{\Lambda_m}< \infty$ and is called (inhomogeneous) Lipschitz space..
\end{defn}
\begin{defn}[$L_p$-Lipschitz space]
For $T \in (0,\infty)$, $p \in (1,\infty]$, $m \in (0,\infty)$, and a measurable function $f(t,x)$ on $(0,T)\times \fR^d$, we denote
\begin{align}
						\label{lp lip norm}
\|f\|_{L_p\left((0,T) ; \Lambda_m\right)}
:= \left(\int_0^T \|f(t,\cdot)\|^p_{\Lambda_m} dt\right)^{1/p}.
\end{align}
We say that $f \in L_p((0,T);\Lambda_m)$ iff
$\|f\|_{L_p\left((0,T) ; \Lambda_m\right)} <\infty$.
\end{defn}
\begin{remark}
\begin{enumerate}[(i)]

\item Since $\|\cdot\|_{\dot \Lambda_m}$ does not recognize polynomials of degree up to order $\lfloor m \rfloor$ and thus it is not a norm.
By identifying two continuous functions whose difference is a polynomial of degree up to order $\lfloor m \rfloor$,
we can regard $\|\cdot\|_{\dot \Lambda_m}$ as a norm.

\item $\Lambda_m$ and $L_p\left((0,T) ; \Lambda_m\right)$ are Banach spaces.

\item From the definition of $\Lambda_m$, one can easily check that for any $ f \in L_p((0,T);\Lambda_m)$ and $0<t \leq T$,
\begin{align}
						\label{e 1029 2}
\left\|\int_0^t f(s,\cdot) ds\right\|_{\Lambda_m}
\leq \int_0^t \| f(s,\cdot) \|_{\Lambda_m} ds.
\end{align}
Moreover, if we consider  the integral $\int_0^t f(s,\cdot) ds$ as Bochner's integral, then 
(\ref{e 1029 2}) is one of simple properties of the Bochner integral. 
\end{enumerate}
\end{remark}
Here is the main result of this paper.
\begin{thm}
					\label{main thm}
Let  $m,T \in (0,\infty)$ and $p \in (1,\infty]$. Then for any $f \in L_p((0,T);\Lambda_m)$, 
there exists a unique solution $u \in L_p((0,T);\Lambda_{\gamma+m})$ to equation (\ref{main eqn}).
Furthermore, for this solution $u$, 
\begin{align}
						\label{main est}
\int_0^T \|u(t,\cdot)\|^p_{\Lambda_{\gamma+m}} dt \leq N \int_0^T \|f(t,\cdot)\|^p_{\Lambda_{m}}dt,
\end{align}
where $N$ depends only on $d$, $p$, $\gamma$, $\nu$, $m$, and $T$.
\end{thm}
The proof of this theorem will be given in Section {\ref{pf main thm}.

\vspace{2mm}

There is a close relation between the Lipschitz space $\Lambda_{n+\alpha}$ and the classical H\"older space the $C^{n+\alpha}$.
We recall the definition of the classical H\"older spaces and introduce a comparison between the Lipschitz space and the H\"older space briefly. 
\begin{defn}[H\"older space]
For $n \in \bZ_+$, $\alpha \in (0,1)$, and $f \in C^n$, we define 
\begin{align}
							\label{lip norm}
\|f\|_{C^{n+\alpha}}:=\|f\|_{L_\infty} + \sum_{|\beta|=n} \sup_{x \in \fR^d}\sup_{h \in \fR^d \setminus \{0\}} \frac{\left|D_h(D^{\beta}(f))(x)\right|}{|h|^\alpha}.
\end{align}
The space of continuous functions $f$ such that $\|f\|_{C^{n+\alpha}}< \infty$ is called H\"older space whose order is $n+\alpha$. 
\end{defn}

\begin{defn}[$L_p$-H\"older space]
For $T \in (0,\infty)$, $p \in (1,\infty]$, $n \in \bZ_+$, $\alpha \in (0,1)$, and a measurable function $f(t,x)$ on $(0,T)\times \fR^d$, we denote
\begin{align}
						\label{lp hol norm}
\|f\|_{L_p\left((0,T) ; C^{n+\alpha}\right)}
:= \left(\int_0^T \|f(t,\cdot)\|^p_{C^{n+\alpha}} dt\right)^{1/p}.
\end{align}
We say that $f \in L_p((0,T);C^{n+\alpha})$ iff
$\|f\|_{L_p\left((0,T) ; C^{n+\alpha}\right)} <\infty$.
\end{defn}

\begin{remark}
From the definitions of the Lipschitz space and the H\"older space, one can easily check the following two properties:
\begin{enumerate}[(i)]
\item For any $\alpha \in (0,\infty)$, $C^{\alpha} \subset \Lambda_{\alpha}$ and the inclusion is continuous i.e. there exists a constant $N$ so that
$$
\|f\|_{\Lambda_\alpha} \leq N \|f\|_{C^\alpha} \qquad \forall f \in C^\alpha.
$$
\item $C^{\alpha} = \Lambda_{\alpha}$ if $\alpha \in (0,1)$.
\end{enumerate}
\end{remark}

\begin{thm}
					\label{main thm 2}
Let $T \in (0,\infty)$, $p \in (1,\infty]$, $n \in \bZ_+$, and $\alpha \in (0,1)$ so that $\gamma+\alpha \notin \bZ_+$. Then for any $f \in L_p((0,T);C^{n+\alpha})$, 
there exists a unique solution $u \in L_p((0,T);C^{\gamma+n+\alpha})$ to equation (\ref{main eqn}).
Furthermore, for this solution $u$, 
\begin{align}
						\label{main est 2}
\int_0^T \|u(t,\cdot)\|^p_{C^{\gamma+n+\alpha}} dt
\leq N \int_0^T \|f(t,\cdot)\|^p_{C^{n+\alpha}}dt,
\end{align}
where $N$ depends only on $d$, $p$, $\gamma$, $\nu$,  $\alpha$, and $T$.
\end{thm}
The proof of this theorem will be given in Section {\ref{pf main thm}.

\mysection{Preliminaries}

We introduce kernels related to the symbol $\psi(t,\xi)$ satisfying (\ref{sym1}) and (\ref{sym2}).
For $s<t$, $a, b \in \fR$, and a multi-index $\alpha$, denote
\begin{align*}
 p_{\alpha,a,b}(s,t,x):= \cF^{-1}\left[ D^\alpha_\xi\left((\psi(t,\xi))^a|\xi|^b\exp\left( \int_s^t \psi(r,\xi)dr \right)\right) \right](x),
\end{align*}
\begin{align*}
\hat q_{\alpha,a,b}(s,t,\xi)
:=  D^\alpha_\xi\left[ \left((t-s)\psi(t,(t-s)^{-1/\gamma}\xi)\right)^a|\xi|^b\exp\left( \int_s^t \psi(r,(t-s)^{-1/\gamma}\xi)dr \right)\right],
\end{align*}
and
\begin{align*}
 q_{\alpha,a,b}(s,t,x):= \cF^{-1}\left[\hat q_{\alpha,a,b}(s,t,\cdot) \right](x).
\end{align*}
In particular, we set
\begin{align*}
p_{a,b}(s,t,x):= p_{0,a,b}(s,t,x), ~\hat q_{a,b}(s,t,\xi):= \hat q_{0,a,b}(s,t,\xi), ~q_{a,b}(s,t,x):= q_{0,a,b}(s,t,x),
\end{align*}
and
$$
p(s,t,x):=p_{0,0,0}(s,t,x).
$$
\begin{remark}
If $a \geq 0$ and $b \geq 0$, then obviously for each $s<t$
$$
(\psi(t,\cdot))^a|\cdot|^b \exp\left( \int_s^t \psi(r,\cdot)dr \right) \in L_1(\fR^d) \cap L_2(\fR^d)
$$
due to (\ref{sym1}) and (\ref{sym2}).  Therefore 
$$
p_{a,b}(s,t,\cdot) \in L_2(\fR^d) \cap L_\infty(\fR^d) 
$$
and $p_{a,b}(s,t,\cdot) \in L_p(\fR^d)$ for all $p \geq 2$.
However, we do not know $p_{a,b}(s,t,\cdot) \in L_p(\fR^d)$ for $p  \in [1,2)$ yet.
\end{remark}

\begin{lemma}
Let $a \in \bZ_+$, $b \in \fR$, and $\alpha$ be a  multi-index so that $|\alpha | \leq d_0 $.
Then there exists a constant $N(d,a,b,\gamma,\nu )$ such that for all $s<t$ and  $\xi \neq 0$,
\begin{align}
								\notag
\left| \hat q_{\alpha,a,b}(s,t,\xi) \right| 
&\leq   N 1_{a=b=0,\alpha \neq 0} \cdot |\xi|^{\gamma - |\alpha|}e^{-\nu|\xi|^\gamma} \\ 
								\label{928 1}
&\quad + N\left(1_{a\neq 0}+1_{b\neq 0}+1_{\alpha = 0} \right)\cdot |\xi|^{\gamma a + b - |\alpha|}e^{-\nu|\xi|^\gamma}.
\end{align}
In particular, 
if either 
(i) $a=b=0$, $\alpha \neq 0$ or
(ii) $\gamma a + b - |\alpha| > -d$ holds then
\begin{align}
						\label{0806 eqn 1}
\sup_{s<t}\int_{\fR^d} \left|\hat q_{\alpha,a,b}(s,t,\xi) \right| d\xi\leq N(d,a,b,\gamma,\nu).
\end{align}
\end{lemma}
\begin{proof}
This is an easy consequence of (\ref{sym1}) and (\ref{sym2}).
\end{proof}

\begin{corollary}
Let $a \in \bZ_+$, $b \in \fR$, and $\alpha$ be a  multi-index so that $|\alpha | \leq d_0 $.
Assume that either (i) $a=b=0$ and $\alpha \neq 0$ or
(ii) $\gamma a + b - |\alpha| > -d$ holds.
Then there exists a constant $N(d,a,b,\gamma,\nu)$ such that
\begin{align}
						\label{eq 1013 1}
\sup_{s<t,x \in \fR^d} |q_{\alpha,a,b}(s,t,x)| \leq N. 
\end{align}
\end{corollary}
\begin{proof}
This is an easy consequence of a property of the Fourier inverse transform and (\ref{0806 eqn 1}).
Indeed,
$$
\sup_{s<t,x \in \fR^d} |q_{\alpha,a,b}(s,t,x)| 
\leq \sup_{s<t} \| \hat q_{\alpha,a,b}(s,t,\cdot) \|_{L_1(\fR^d)} \leq N.
$$
The corollary is proved. 
\end{proof}
\begin{lemma}
Let $a \in \bZ_+$, $b \in \fR$, and $\alpha$ be a multi-index so that $|\alpha | \leq d_0$.
Then there exists a constant $N$ such that for all $\varepsilon \in (0,1)$, $s<t$, and $\xi \neq 0$,
\begin{align}
							\notag
&\int_{|\xi| \geq \varepsilon}  \left|  \hat q_{\alpha,a,b}(s,t,\xi) \right|^2 d\xi \\
							\notag
&\leq N 1_{a=b=0,\alpha \neq 0} \cdot \left(1+ \varepsilon^{2(\gamma-|\alpha|) +d} + 1_{2(\gamma-|\alpha|)=-d} \cdot \ln \varepsilon^{-1} \right) \\
							\label{819 1}
&\quad +N \left(1_{a\neq 0}+1_{b\neq 0}+1_{\alpha = 0} \right) \left(1+ \varepsilon^{2(\gamma a+b-|\alpha|) +d} + 1_{2(\gamma a+b-|\alpha|)=-d} \cdot \ln \varepsilon^{-1} \right),
\end{align}
where $N=N(d,a,b,\gamma,\nu)$. 
In particular, if either (i) $a=b=0$, $\alpha \neq 0$, $2(\gamma-|\alpha|)> -d$ or 
(ii) $2(\gamma a + b - |\alpha|)>-d$ holds, then
\begin{align}
						\label{p eq}
\sup_{s<t}\int_{\fR^d}\left| \hat q_{\alpha,a,b}(s,t,\xi)\right|^2 d\xi
\leq N(d,a,b,\gamma,\nu).
\end{align}
\end{lemma}
\begin{proof}
By (\ref{928 1}),
\begin{align*}
&\int_{|\xi| \geq \varepsilon} \left|\hat q_{\alpha,a,b}(s,t,\xi)\right|^2 d\xi \\
&\leq \int_{|\xi| \geq 1} \left|\hat q_{\alpha,a,b}(s,t,\xi)\right|^2 d\xi
+\int_{\varepsilon \leq |\xi| \leq 1} \left|\hat q_{\alpha,a,b}(s,t,\xi)\right|^2 d\xi \\
&\leq N  \\
&\quad +N\int_{\varepsilon \leq |\xi| \leq 1} \Big( 1_{a=b=0,\alpha \neq 0} \cdot |\xi|^{2(\gamma - |\alpha|)}  
 + \left(1_{a\neq 0}+1_{b\neq 0}+1_{\alpha = 0} \right)\cdot |\xi|^{2(\gamma a + b - |\alpha|)}\Big) d\xi\\
&\leq N\left(1+ 1_{a=b=0,\alpha \neq 0} \cdot \varepsilon^{2(\gamma-|\alpha|) +d} + 1_{2(\gamma-|\alpha|) =-d} \cdot \ln \varepsilon^{-1} \right) \\
&\quad +N\left(1_{a\neq 0}+1_{b\neq 0}+1_{\alpha = 0} \right)\cdot\left(1+ \varepsilon^{2(\gamma a + b-|\alpha|)+d} + 1_{2(\gamma a +b-|\alpha|)=-d} \cdot \ln \varepsilon^{-1} \right).
\end{align*}
The lemma is proved. 
\end{proof}

\begin{lemma}
						\label{lem 1013 1}
Let $a \in \bZ_+$ and $b \geq 0$ be constants such that 
either (i) $a=b=0$ or (ii) $\gamma a + b > 0$ holds.
Assume 
\begin{align}
						\label{e 1101 1}
\begin{cases}
&0<\delta < \frac{1}{2} \wedge \left(\gamma a +b\right) \quad \text{if}~ \gamma a +b >0 \\
&0<\delta < \frac{1}{2} \wedge \gamma \quad \text{if}~a= b =0 .
\end{cases}
\end{align}
Then there exists a constant $N(d,a,b,\gamma,\nu, \delta)$ such that
\begin{align}
						\label{eq 1013 2}
\sup_{s<t} \int_{\fR^d} \left| |x|^{\frac{d}{2}+\delta} q_{a,b}(s,t,x)\right|^2dx \leq N.
\end{align}
\end{lemma}
\begin{proof}
Set $$
\delta_0:= \frac{d}{2} +1 -d_0 + \delta=\frac{d}{2}-\left\lfloor\frac{d}{2}\right\rfloor+\delta \in (0,1)
$$
and recall the definition of the fractional Laplacian operator 
$$
(-\Delta)^{\delta_0/2}f(x) = \cF^{-1}\left[ |\xi|^{\delta_0} \cF[f](\xi) \right] (x).
$$
Then by properties of the Fourier inverse transform and Plancherel's theorem, 
\begin{align*}
&\sup_{s<t} \int_{\fR^d} \left| |x|^{\frac{d}{2}+\delta} q_{a,b}(s,t,x)\right|^2dx  \\
&\leq N\sum_{j=1}^d \sup_{s<t} \int_{\fR^d} \left| |x|^{\frac{d}{2}-\left\lfloor\frac{d}{2}\right\rfloor+\delta} (ix^j)^{\left\lfloor\frac{d}{2}\right\rfloor}q_{a,b}(s,t,x)\right|^2dx \\ 
&= N\sum_{j=1}^d \sup_{s<t} \int_{\fR^d} \left| |x|^{\frac{d}{2}-\left\lfloor\frac{d}{2}\right\rfloor+\delta} \cF^{-1}\left[\hat q_{\left\lfloor\frac{d}{2}\right\rfloor e_j,a,b}(s,t,\cdot)\right](x)\right|^2dx \\ 
&= N\sum_{j=1}^d \sup_{s<t} \int_{\fR^d} \left| (-\Delta)^{\delta_0/2}\left[\hat q_{\left\lfloor\frac{d}{2}\right\rfloor e_j,a,b}(s,t,\cdot)\right](x)\right|^2dx,
\end{align*}
where $e_j$ $(j=1,\ldots,d)$ is the standard orthonormal basis on $\fR^d$, i.e. $e_j$ is the vector in $\fR^d$ whose $j$-th coordinate is 1 and the others are zero. 
It is well-known that the fractional Laplacian operator has the integral representation 
\begin{align*}
&(-\Delta)^{\delta_0/2}\left[\hat q_{\left\lfloor\frac{d}{2}\right\rfloor e_j,a,b}(s,t,\cdot)\right](x) \\
&=N  \int_{\fR^d} \frac{\hat q_{\left\lfloor\frac{d}{2}\right\rfloor 
e_j,b}(s,t,x+y)-\hat q_{\left\lfloor\frac{d}{2}\right\rfloor e_j,a,b}(s,t,x)}{|y|^{d+\delta_0}} dy \\
&= N\left(\cI_1(s,t,x) + \cI_2(s,t,x) +\cI_3(s,t,x)\right),
\end{align*}
where
\begin{align*}
\cI_1(s,t,x)
=\int_{|y| \geq 1} \frac{\hat q_{\left\lfloor\frac{d}{2}\right\rfloor e_j,a,b}(s,t,x+y)-\hat q_{\left\lfloor\frac{d}{2}\right\rfloor e_j,a,b}(s,t,x)}{|y|^{d+\delta_0}} dy,
\end{align*}
\begin{align*}
\cI_2(s,t,x)
=\int_{|x|/2 \leq |y| <1} \frac{\hat q_{\left\lfloor\frac{d}{2}\right\rfloor e_j,a,b}(s,t,x+y)-\hat q_{\left\lfloor\frac{d}{2}\right\rfloor e_j,a,b}(s,t,x)}{|y|^{d+\delta_0}} dy,
\end{align*}
and
\begin{align*}
\cI_3(s,t,x)
=\int_{|y| < 1, |y| < |x|/2} \frac{\hat q_{\left\lfloor\frac{d}{2}\right\rfloor e_j,a,b}(s,t,x+y)-\hat q_{\left\lfloor\frac{d}{2}\right\rfloor e_j,a,b}(s,t,x)}{|y|^{d+\delta_0}} dy.
\end{align*}
First we estimate $\cI_1$.
By  Minkowski's inequality and (\ref{p eq}),
\begin{align*}
\sup_{s<t}\|\cI_1(s,t,\cdot)\|^2_{L_2}
\leq 2\sup_{s<t}\left\|\hat q_{\left\lfloor\frac{d}{2}\right\rfloor e_j,a,b}(s,t,\cdot)\right\|^2_{L_2} 
\left(\int_{|y| \geq 1} |y|^{-d-\delta_0} dy \right)^2
\leq N.
\end{align*}
Next we estimate $\cI_2$. 
By the the fundamental theorem of calculus and (\ref{928 1}),
\begin{align*}
&\cI_2(s,t,x) \\
&\leq \sum_{|\alpha| = d_0}\int_{|x|/2 \leq |y| < 1} |y| ^{1-d-\delta_0} \int_0^1 \left| \hat q_{\alpha,a,b}(s,t,x + \theta y)  \right| d\theta dy   \\
&\leq N
\int_{ |y| <1} 1_{|x| \leq 2|y|} |y| ^{1-d-\delta_0}   \times \\
&\qquad \qquad \qquad \int_0^1 \left(1_{a=b=0} \cdot |x+\theta y|^{\gamma-d_0}
+\left(1_{a\neq 0}+ 1_{b \neq 0}\right) \cdot| x+ \theta y|^{\gamma a + b-d_0}   \right) d\theta dy.
\end{align*}
Noe that $\cI_2=0$ if $|x| >2$.
Therefore by Minkowski's inequality,
\begin{align*}
&\|\cI_2(s,t,\cdot)\|_{L_2(\fR^d)} \\
&\leq 
N \int_{ |y| <1} |y|^{\frac{d}{2}}
\left(1_{a=b=0} \cdot |y|^{\gamma-d_0-d-\delta_0+1}+\left(1_{a\neq 0} +1_{b \neq 0}\right) \cdot|y|^{\gamma a + b-d_0-d-\delta_0+1}   \right) dy   \\
&\leq N,
\end{align*}
where (\ref{e 1101 1}) is used in the last inequality.
It only remains to estimate $\cI_3$. 
By the fundamental theorem of calculus, Minkowski's inequality, (\ref{819 1}), and (\ref{e 1101 1}),
\begin{align*}
&\|\cI_3(s,t,\cdot)\|^2_{L_2(\fR^d)} \\
&\leq \left\|\sum_{|\alpha|=d_0}\int_{|y| < 1, |y| < |\cdot|/2} |y| \int_0^1 \left|\frac{\hat q_{\alpha ,a,b}(s,t,\cdot+ \theta y)}{|y|^{d+\delta_0}} \right|d\theta dy \right\|^2_{L_2(\fR^d)} , \\
&\leq N\left|\int_{|y| < 1,} |y|^{1-d-\delta_0} 
\left( \int_{|x| \geq |y|} \left|\hat q_{d_0 e_j,a,b}(s,t,x) \right|^2 dx \right)^{1/2} dy \right|^2 \\
&\leq N\left|\int_{|y| < 1,} |y|^{1-d-\delta_0} \cdot
\left(1_{a=b=0} \cdot |y|^{\gamma- d_0 +\frac{d}{2}} + \left(1_{a\neq0} +1_{b\neq 0}\right)\cdot |y|^{\gamma a + b -d_0 +\frac{d}{2}} \right)  dy \right|^2  \\
& \leq N.
\end{align*}
Therefore, the lemma is proved. 
\end{proof}

\begin{corollary}
						\label{cor 1014 1}
Let $a \in \bZ_+$ and $b \in \fR$ such that 
$a=b=0$ or 
$$
\gamma a + b  > 0.
$$
Then there exists a constant $N(d,a,b,\gamma,\nu)$ such that for all $s<t$, 
\begin{align*}
 \int_{\fR^d}\left|p_{a,b}(s,t,x)\right|dx \leq N (t-s)^{-a-b/\gamma}.
\end{align*}
In particular,
\begin{align*}
\sup_{s<t} \int_{\fR^d}\left|p(s,t,x)\right|dx \leq N,
\end{align*}
\begin{align*}
\sup_{s<t} \int_{\fR^d}\left|p_{0,\gamma}(s,t,x)\right|dx \leq N(t-s)^{-1},
\end{align*}
and
\begin{align*}
\sup_{s<t} \int_{\fR^d}\left|p_{1,\gamma}(s,t,x)\right|dx \leq N(t-s)^{-2}.
\end{align*}
\end{corollary}
\begin{proof}
Set
\begin{align}
\delta=
\begin{cases}
& \frac{1}{4} \wedge \frac{\gamma}{2} \quad \text{if}~a=b=0\\
& \frac{1}{4} \wedge \frac{\gamma a + b}{2} \quad \text{if}~\gamma a + b > 0.
\end{cases}
\end{align}
By H\"older's inequality, (\ref{eq 1013 1}), and (\ref{eq 1013 2}), 
\begin{align*}
\sup_{s<t} \int_{\fR^d}\left|q_{a,b}(s,t,x)\right|dx 
&\leq  N\left(1+\sup_{s<t} \int_{|x| \geq 1} \left|q_{a,b}(s,t,x)\right|dx\right) \\
&\leq  N\left[1+ \left(\sup_{s<t} \int_{|x| \geq 1} \left| |x|^{\frac{d}{2}+\delta} q_{a,b}(s,t,x)\right|^2dx\right)^{1/2} \right]
\leq N.
\end{align*}
Therefore 
\begin{align*}
 \int_{\fR^d}\left|p_{a,b}(s,t,x)\right|dx
=(t-s)^{-a-b/\gamma} \int_{\fR^d}\left|q_{a,b}(s,t,x)\right|dx
\leq N(t-s)^{-a-b/\gamma}.
\end{align*}
The corollary is proved. 
\end{proof}

\mysection{$L_\infty(\Lambda_{\gamma+m})$-estimate}

In this section, we introduce Littlewood-Paley operators, which  play a crucial role in modern analysis and are very helpful to characterize diverse function  spaces such as Sobolev space, Besov space, Lipschitz space, Hardy space, and Triebel-Lizorkin space. 
Applying this powerful characterization,  we  obtain an $L_\infty(\Lambda_{\gamma+m})$-estimate.
\smallskip

Choose a nonnegative function $\eta \in C_c^\infty(\fR^d)$ such that $\eta(\xi)=1$ for all $|\xi| \leq 1$ and
$\eta(\xi)=0$ for all $|\xi| \geq 2$.
For $n \in \bZ$, define $\delta_n(\xi)= \eta(2^{-n}\xi) -\eta(2^{-n+1}\xi)$.
Then obviously
$\delta_n$ has a support in $\left(2^{n-1} , 2^{n+1}\right)$ and
\begin{align}
						\label{sum 1}
 \sum_{n=-\infty}^\infty \delta_n(\xi)
=\eta(\xi) + \sum_{n=1}^\infty \delta_n(\xi)=1.
\end{align}
Denote
$$
\Phi(x) := \cF^{-1}[\eta(\xi)](x)
$$
and
$$
\Psi_n(x) := \cF^{-1}[\delta_n(\xi)](x)=\cF^{-1}[\eta(2^{-n}\xi)](x)-\cF^{-1}[\eta(2^{-n+1}\xi)](x).
$$
For a function $f \in C(\fR^d)$, we define the Littlewood-Paley operators as
$$
S_0(f)(x) := \Phi \ast f(x)
$$
and 
$$
\Delta_n f(x) := \Psi_n \ast f(x),
$$
where $\ast$ denotes the convolution on $\fR^d$, i.e. 
$$
f \ast g (x):=f(\cdot) \ast g(\cdot) (x) := \int_{\fR^d} f(x-y)g(y)dy.
$$
Here is the Littlewood-Paley characterization for the Lipschitz space.
\begin{thm}
						\label{norm thm}
Let $m>0$ and $f \in C(\fR^d)$. Then
\begin{enumerate}[(i)]
\item
$$
\|f\|_{\dot \Lambda_m} < \infty
$$
if and only if
\begin{align}
						\label{homo cha}
 \sup_{n  \in \bZ} 2^{nm}\|\Delta_n(f)\|_{L_\infty} <\infty
\end{align}

\item
$$
\|f\|_{\Lambda_m} < \infty
$$
if and only if
\begin{align}
						\label{e 1026 1}
\|S_0(f)\|_{L_\infty} + \sup_{n \geq 1} 2^{nm}\|\Delta_n(f)\|_{L_\infty} <\infty.
\end{align}
\end{enumerate}
Moreover, there exists a constant $N(d,m)$ so that
\begin{align*}
N^{-1}\|f\|_{ \dot \Lambda_m} 
\leq   \sup_{n \in \bZ} 2^{nm}\|\Delta_n(f)\|_{L_\infty}
\leq N\|f\|_{\dot \Lambda_m}
\end{align*}
and
\begin{align*}
N^{-1}\|f\|_{\Lambda_m} 
\leq  \|S_0(f)\|_{L_\infty} + \sup_{n \geq 1} 2^{nm}\|\Delta_n(f)\|_{L_\infty}
\leq N\|f\|_{\Lambda_m}.
\end{align*}
\end{thm}
\begin{proof}
For the proof of this theorem, see \cite[Theorems 6.3.6 - 6.3.7]{grafakos2009modern}.
\end{proof}
Recall that
\begin{align*}
p(s,t,x):= \cF^{-1}\left[ \exp\left( \int_s^t \psi(r,\xi)dr \right) \right](x)
\end{align*}
and by Corollary \ref{cor 1014 1},
\begin{align*}
\sup_{s<t} \| p(s,t,\cdot)\|_{L_1(\fR^d)}  < \infty.
\end{align*}
Thus for a bounded measurable function $f$ on $[0,T] \times \fR^d$,
one can define
\begin{align*}
\cG f(t,x) := 
\int_0^t p(s,t,\cdot) \ast f(s,\cdot)(x) ds \quad \forall (t,x) \in [0,T] \times \fR^d
\end{align*}
and obtain
\begin{align}
							\notag
\sup_{t \in [0,T]}\|\cG f(t,\cdot) \|_{L_\infty(\fR^d)} 
&\leq  \int_0^T  \|p(s,t,\cdot)\| _{L_1(\fR^d)} \| f(s,\cdot)\|_{L_\infty(\fR^d)} ds \\
							\label{l infty bounded}
&\leq  N(d,\gamma,\nu,T) \sup_{s \in [0,T]} \| f(s,\cdot)\|_{L_\infty(\fR^d)}.
\end{align}

\begin{remark}
Let $f \in L_p\left((0,T); \dot \Lambda_{m} \right)$,  $f \in L_p\left((0,T); \Lambda_{m} \right)$, or $f \in L_p\left((0,T); C^{n+\alpha}\right)$, where
$p \in (1,\infty]$, $m > 0$, $n \in \bZ_+$, and $\alpha \in (0,1)$.
Then for any $(t,x) \in (0,T) \times \fR^d$, $\cG f(t,x)$ is well-defined 
and  for each $t \in (0,T)$, $\cG(t,x)$ is continuous with respect to $x$.
Moreover, we have
\begin{align*}
\|\cG f(t,\cdot)\|_{\dot \Lambda_m}
\leq Nt^{1/q}\|f\|_{L_p\left((0,t); \dot \Lambda_{m} \right)},
\end{align*}
\begin{align*}
\|\cG f(t,\cdot)\|_{\Lambda_m}
\leq Nt^{1/q}\|f\|_{L_p\left((0,t); \Lambda_{m} \right)},
\end{align*}
and
\begin{align}
						\label{e 1023 2}
\|\cG f(t,\cdot)\|_{C^{n+\alpha}}
\leq Nt^{1/q}\|f\|_{L_p\left((0,t); C^{n+\alpha} \right)},
\end{align}
where $N=N(d,\gamma, \nu)$ and $q$ is the H\"older conjugate of $p$, i.e. $1/p +1/q =1$.
\end{remark}
\begin{lemma}
						\label{l 1017 3}
Let $m>0$ and $f \in C(\fR^d)$. Then there exists a positive constant $N(d,\gamma, \nu, m)$ so that 
\begin{enumerate}[(i)]
\item for all $s<t$,
\begin{align*}
\|p(s,t,\cdot) \ast f(\cdot) \|_{\dot \Lambda_{\gamma+m}} 
\leq N (t-s)^{-1}  \|f\|_{\dot \Lambda_m}
\end{align*}
and
\begin{align}
						\label{e 0711 1}
\|p(s,t,\cdot) \ast f(\cdot) \|_{\Lambda_{\gamma+m}} 
\leq N\left(1+(t-s)^{-1}\right) \|f\|_{\Lambda_m}
\end{align}
\item for all $s \vee t_0 < t$,
\begin{align*}
\| \left(p(s,t,\cdot)-p(t_0,t,\cdot) \right) \ast f(\cdot) \|_{\dot \Lambda_{\gamma+m}} 
\leq N|s-t_0|\left(\left(t-(s\vee t_0) \right)^{-2}\right) \|f\|_{\dot \Lambda_m},
\end{align*}
and
\begin{align}
						\label{e 0711 2}
\| \left(p(s,t,\cdot)-p(t_0,t,\cdot) \right) \ast f(\cdot) \|_{\Lambda_{\gamma+m}} 
\leq N|s-t_0|\left(1+\left(t-(s\vee t_0) \right)^{-2}\right) \|f\|_{\Lambda_m},
\end{align}
\end{enumerate}
where $N=N(d,\gamma,\nu,m)$.
\begin{proof}
Due to the similarity of the proof, we only prove inhomogeneous type estimates (\ref{e 0711 1}) and (\ref{e 0711 2}). 
Moreover, we may assume that $\|f\|_{\Lambda_m} < \infty$ without loss of generality. 
\smallskip

(i) Choose a $\zeta \in C^\infty_c(\{\xi \in \fR^d : 2^{-2} \leq |\xi| \leq 2^2\})$ so that 
$\zeta(\xi) =1$ if $2^{-1} \leq |\xi| \leq 2$.
Recall $\delta_n(\xi) := \eta(2^{-n}\xi) - \eta(2^{-n+1}\xi)$,
$\eta(\xi)=1$ for all $|\xi| \leq 1$ and $\eta(\xi)=0$ for all $|\xi| \geq 2$. 
Thus $\delta_n(\xi)$ has a support in $\{\xi \in \fR^d : 2^{n-1} \leq |\xi| \leq 2^{n+1} \}$ and $\delta_n(\xi) = \delta_n(\xi) \zeta(2^{-n}\xi)$.
Observe that 
\begin{align}
						\label{e 1017 3}
S_0 \left( p(s,t,\cdot) \ast f(\cdot) \right)(x)
&=   p(s,t,\cdot) \ast  S_0(f)(\cdot)(x) 
\end{align}
and
\begin{align}
						\notag
\Delta_n(\left(p(s,t,\cdot) \ast f(\cdot) \right)(x)
&=\cF^{-1}\left[ \zeta(2^{-n}\xi)\delta_n(\xi) \exp\left( \int_s^t \psi(r,\xi)dr \right) \right](\cdot) \ast f(\cdot)(x)  \\
						\notag
&= \cF^{-1}\left[ \zeta(2^{-n}\xi) |\xi|^{-\gamma}\right](\cdot) \ast p_{0,\gamma}(s,t,\cdot) \ast \Psi_n(\cdot) \ast f(\cdot)(x)  \\
						\label{e 1017 4}
&= \cF^{-1}\left[ \zeta(2^{-n}\xi) |\xi|^{-\gamma}\right](\cdot) \ast  p_{0,\gamma} (s,t,\cdot) \ast  \Delta_n(f)(\cdot)(x) ds.
\end{align}
Therefore by Theorem \ref{norm thm}, Young's inequality,  and Corollary \ref{cor 1014 1},
\begin{align*}
&\|p(s,t,\cdot) \ast f(\cdot) \|_{\Lambda_{\gamma+m}} \\
&\leq N \left( 1+ \left(t-s \right)^{-1}
 \sup_{n \in \bN} 2^{n\gamma} \left\|\cF^{-1}\left[ \zeta(2^{-n}\xi) |\xi|^{-\gamma}\right] \right\|_{L_1(\fR^d)} \right)\|f\|_{\Lambda_m} .
\end{align*}
It only remains to observe that
$$
\left\|\cF^{-1}\left[ \zeta(2^{-n}\xi) |\xi|^{-\gamma}\right] \right\|_{L_1(\fR^d)} 
\leq N 2^{-n\gamma}.
$$

The proofs of (ii) is similar to (i). We only remark that by the mean-value theorem
$$
p(s,t,\cdot)-p(t_0,t,\cdot) 
= \frac{\partial}{\partial s} p( \theta s + (1-\theta)t_0,t,\cdot) \qquad (\theta \in [0,1])
$$
and by Corollary \ref{cor 1014 1},
\begin{align*}
\left\|\frac{\partial}{\partial s} p_{0,\gamma}( \theta s + (1-\theta)t_0,t,\cdot) \right\|_{L_1(\fR)}
&=\left\|p_{1,\gamma}( \theta s + (1-\theta)t_0,t,\cdot) \right\|_{L_1(\fR)} \\
&\leq N  \left| t-  \left( \theta s + (1-\theta)t_0 \right) \right|^{-2} .
\end{align*}
The lemma is proved. 
\end{proof}
Since $\int_0^t (t-s)^{-1}ds = \infty$, Lemma \ref{l 1017 3} is not enough to  show that $\cG$ is a bounded operator from  $L_\infty((0,T) ; \Lambda_{m})$ into $L_\infty((0,T) ; \Lambda_{\gamma+m})$. 
We need more delicate estimates. 
\end{lemma}

\begin{lemma}
						\label{l 1017 1}
Let $\zeta(\xi) \in C^\infty_c\left( \{\xi \in \fR^d : 2^{-2} \leq |\xi| \leq 2^2\} \right)$.
Then there exists a constant $c>0$ and $N>0$ so that for all $s<t$ and $n \in \bN$, 
\begin{align*}
\left\| \cF^{-1}\left[ \zeta(2^{-n} \xi)\exp\left( \int_s^t \psi(r,\xi)dr \right) \right](\cdot) \right\|_{L_1(\fR^d)}
\leq N e^{-c(t-s)2^{n\gamma}},
\end{align*}
where $c$ and $N$ depend only on $d$, $\gamma$, $\nu$, and $\zeta$. 
\end{lemma}
\begin{proof}
It suffices to show that for all $j = 1 ,\ldots, d$,
\begin{align}
							\notag
&\int_{\fR^d} \left(\cF^{-1}\left[ \zeta(\xi)\exp\left( \int_s^t \psi(r,2^n\xi)dr \right) \right](x) \right)^2 dx  \\
							\label{e 1015 1}
&\quad+ \int_{\fR^d} \left( \cF^{-1}
\left[ D_{\xi}^{d_0e_j}\left(\zeta(\xi)\exp\left( \int_s^t \psi(r,\xi)dr \right) \right)\right](x) \right)^2 dx 
\leq N e^{-2c(t-s)2^{n\gamma}}.
\end{align}
Indeed, if (\ref{e 1015 1}) holds, then by the change of variable, H\"older's inequality, and the property of the Fourier inverse transform that
$$
x^j\cF^{-1}[g(\xi)](x) 
=i\cF^{-1}[ D^j\left(g(\xi) \right)](x),
$$
we have
\begin{align*}
&\left\| \cF^{-1}\left[ \zeta(2^{-n}\xi)\exp\left( \int_s^t \psi(r,\xi)dr \right) \right](\cdot) \right\|_{L_1(\fR^d)} \\
&=\left\| \cF^{-1}\left[ \zeta(\xi)\exp\left( \int_s^t \psi(r,2^n\xi)dr \right) \right](\cdot) \right\|_{L_1(\fR^d)} \\
&\leq 
N\left(\int_{\fR^d}\left| (1+|x|)^{d_0}\cF^{-1}\left[ \zeta(\xi)\exp\left( \int_s^t \psi(r,2^n\xi)dr \right) \right](x) \right|^2 dx \right)^{1/2} \\
&\leq N\sum_{j=1}^d \left(\int_{\fR^d}\left| (1+|x^j|^{d_0})\cF^{-1}\left[ \zeta(\xi)\exp\left( \int_s^t \psi(r,2^n\xi)dr \right) \right](x) \right|^2 dx \right)^{1/2}  \\
&\leq N\left(\int_{\fR^d}\left| \cF^{-1}\left[ \zeta(\xi)\exp\left( \int_s^t \psi(r,2^n\xi)dr \right) \right](x) \right|^2 dx \right)^{1/2}  \\
&\quad +N\sum_{j=1}^d \left(\int_{\fR^d}\left|\cF^{-1}\left[ D_{\xi}^{d_0e_j}
\left(\zeta(\xi)\exp\left( \int_s^t \psi(r,2^n\xi)dr \right) \right)\right](x) \right|^2 dx \right)^{1/2}  \\
&\leq N  e^{-c(t-s)2^{n\gamma}}.
\end{align*}
Due to (\ref{sym1}), and (\ref{sym2}), and the assumption that $\zeta$ has a compact support in  $\{\xi \in \fR^d : 2^{-2} \leq |\xi| \leq 2^2\}$, 
there exist positive constants $N$ and $c$ so that for all $s<t$, $n \in \bN$, $\xi \in \fR^d \setminus \{0\}$, and multi-index $|\alpha|  \leq d_0$,
\begin{align*}
\left| D^\alpha \left(\zeta(\xi)\exp\left( \int_s^t \psi(r,2^n\xi)dr \right)  \right) \right|
\leq N1_{2^{-2} \leq |\xi| \leq 2^2} \cdot e^{-c(t-s)2^{n \gamma}}.
\end{align*}
Therefore by Plancherel's theorem, the left term of (\ref{e 1015 1}) is less than or equal to
\begin{align*}
N \int_{\fR^d}1_{2^{-2} \leq |\xi| \leq 2^2}e^{-2c(t-s)2^{n \gamma}} d\xi
\leq N e^{-2c(t-s)2^{n \gamma}}.
\end{align*}
The lemma is proved. 
\end{proof}
For a function $f(t,x)$, we use the notation that
$S_0(f(t,x)):=S_0f(t,\cdot)(x)$
and $\Delta_n(f(t,x)) := \Delta(f(t,\cdot))(x)$. Here is the main result of this section. 
\begin{thm}
						\label{t 1018 1}
Let $m>0$ and $T \in (0,\infty)$. Then $\cG$ is a bounded operator  from  $L_\infty((0,T) ; \dot \Lambda_{m})$ into $L_\infty((0,T) ; 
\dot \Lambda_{\gamma+m})$ and from  $L_\infty((0,T) ; \Lambda_{m})$ into $L_\infty((0,T) ; \Lambda_{\gamma+m})$.
More precisely, there exist positive constants $N_1$ and $N_2$ so that
\begin{align*}
\|\cG f\|_{L_\infty((0,T) ; \dot \Lambda_{\gamma+m})} 
\leq N_1 \|f\|_{L_\infty((0,T) ; \dot  \Lambda_{m})} \quad \forall f \in L_\infty((0,T) ; \dot \Lambda_{m}),
\end{align*}
and
\begin{align*}
\|\cG f\|_{L_\infty((0,T) ; \Lambda_{\gamma+m})} \leq N_2 \|f\|_{L_\infty((0,T) ; \Lambda_{m})} \quad \forall f \in L_\infty((0,T) ; \Lambda_{m}),
\end{align*}
where $N_1=N_1(d,\gamma,\nu,m)$ and $N_2=N_2(d,\gamma,\nu,m,T)$.
\end{thm}
\begin{proof}
Choose a $\zeta \in C^\infty_c(\{\xi \in \fR^d : 2^{-2} \leq |\xi| \leq 2^2\})$ so that 
$\zeta(\xi) =1$ if $2^{-1} \leq |\xi| \leq 2$ as in the proof of Lemma \ref{l 1017 3}. 
Similarly to (\ref{e 1017 3}) and (\ref{e 1017 4}), we have
\begin{align}
						\notag
S_0 \left( \cG f(t,\cdot)\right)(x)
&= \int_0^t \cF^{-1}\left[ \eta(\xi) \exp\left( \int_s^t \psi(r,\xi)dr \right) \right](\cdot) \ast f(s,\cdot)(x) ds \\
						\notag
&= \int_0^t  p(s,t,\cdot) \ast \Phi(\cdot) \ast f(s,\cdot)(x) ds \\
						\label{e 1017 1}
&= \int_0^t  p(s,t,\cdot) \ast  S_0(f)(s,\cdot)(x) ds 
\end{align}
and
\begin{align}
						\notag
\Delta_n(\cG f(t,\cdot))(x)
&= \int_0^t \cF^{-1}\left[ \zeta(2^{-n}\xi)\delta_n(\xi) \exp\left( \int_s^t \psi(r,\xi)dr \right) \right](\cdot) \ast f(s,\cdot)(x) ds \\
						\notag
&= \int_0^t  p^{\zeta_n} (s,t,\cdot) \ast \Psi_n(\cdot) \ast f(s,\cdot)(x) ds \\
						\label{e 1017 2}
&= \int_0^t  p^{\zeta_n} (s,t,\cdot) \ast  \Delta_n(f)(s,\cdot)(x) ds,
\end{align}
where 
$$
p^{\zeta_n} (s,t,x)
:=\cF^{-1}\left[ \zeta(2^{-n} \xi)\exp\left( \int_s^t \psi(r,\xi)dr \right) \right](x).
$$
Thus by Theorem \ref{norm thm}, (\ref{e 1017 1}), and (\ref{e 1017 2}),
\begin{align*}
\|\cG f(t,\cdot)\|_{ \dot \Lambda_{\gamma+m}} 
&\leq  N \sup_{n  \in \bZ} \int_0^t 2^{n (m+\gamma)} \sup_{x \in \fR^d}\left| p^{\zeta_n}(t,s,\cdot) \ast \Delta_n(f)(s,\cdot)(x) \right|ds\\
&\leq N\sup_{n \in \bZ} \int_0^t \int_{\fR^d} \left(2^{n\gamma} |p^{\zeta_n}(t,s,x)| \right) dx ds \|f\|_{L_\infty((0,T) ; \dot \Lambda_m)}
\end{align*}
and
\begin{align*}
\|\cG f(t,\cdot)\|_{\Lambda_{\gamma+m}} 
&\leq N \int_0^t   \sup_{x \in \fR^d}\left|p(s,t,\cdot) \ast S_0f(s,x) \right|ds \\
&\quad + N \sup_{n \geq 1} \int_0^t 2^{n (m+\gamma)} \sup_{x \in \fR^d}\left| p^{\zeta_n}(t,s,\cdot) \ast \Delta_n(f)(s,\cdot)(x) \right|ds\\
&\leq N\sup_{n \geq 1} \int_0^t \int_{\fR^d} \left(|p(s,t,x)|+ 2^{n\gamma} |p^{\zeta_n}(t,s,x)| \right) dx ds \|f\|_{L_\infty((0,T) ; \Lambda_m)}.
\end{align*}
Therefore it suffices to find positive constants $N_1(d,\gamma, \nu,m)$ and $N_2(d,\gamma,\nu,m,T)$ so that
$$
\sup_{t \in [0,T]}\int_0^t \int_{\fR^d} |p(s,t,x)| dx ds \leq N_2
$$
and
$$
\sup_{t \in [0,T]} \sup_{n  \in \bZ} 2^{n\gamma}  \int_0^t \int_{\fR^d}  |p^{\zeta_n}(t,s,x)| dx ds \leq N_1
$$
But this is an easy consequence of  Corollary \ref{cor 1014 1} and Lemma \ref{l 1017 1}.
The theorem is proved. 
\end{proof}

\mysection{Proof of main theorems}
											\label{pf main thm}

To prove that $\cG$ is a bounded operator from $L_p((0,T) ;\dot  \Lambda_{m})$ into $L_p((0,T) ; \dot  \Lambda_{\gamma+m})$ with $p \in (1,\infty)$, we use the Marcinkiewicz interpolation theorem.
Since it is already shown that $\cG$ is a bounded operator from  $L_\infty((0,T) ; \dot \Lambda_{m})$ into $L_\infty((0,T) ; \dot \Lambda_{\gamma+m})$, we only need to check that $\cG$ satisfies the weak type $(1,1)$-estimate.
\begin{lemma}
					\label{l w11}
Let $T>0$ and $f \in L_1\left((0,T);  \dot  \Lambda_m\right) \cap L_\infty\left((0,T);  \dot \Lambda_m\right)$.
Then there exists a positive constant $N$ so that for  all $\lambda >0$,
\begin{align}
						\label{e 1018 1}
\lambda |\{ t \in (0,T) : \|\cG f(t,\cdot)\|_{  \dot \Lambda_{m+\gamma}} >\lambda \}|
\leq N\int_0^T\|f(t,\cdot)\|_{  \dot  \Lambda_m}dt,
\end{align}
where $N=N(d,\gamma, \nu, m)$. 
\end{lemma}
\begin{proof}
It suffices to find positive constants $N_1, N_2$ which only depend on $d$, $\gamma$, $\nu$, and $m$ so that
for all $\lambda>0$,
\begin{align}
						\label{e 1018 2}
\lambda |\{ t \in (0,T) : \|\cG f(t,\cdot)\|_{ \dot \Lambda_{m+\gamma}} > N_1 \lambda \}|
\leq N_2\int_0^T\|f(t,\cdot)\|_{ \dot \Lambda_m}dt,
\end{align}
If we consider $N_1f(t,x)$ instead of $f$ in (\ref{e 1018 2}), then (\ref{e 1018 1}) is obtained.

Consider a class of dyadic cubes in $\fR$ such that
$$
Q_{k,l} := \left[2^kl , 2^k(l+1) \right) \qquad k.l \in \bZ
$$
and denote
$$
Q^\ast_{k,l} := \left[2^k(l -1) , 2^k(l+2) \right) 
$$
For a fuction $ \bar f(t):=1_{0<t<T}\|f(t,\cdot)\|_{\dot \Lambda_m}$, we apply the Calder\'on-Zygmund decomposition (for instance, see 
\cite[Theorem 4.3.1]{grafakos2008classical}).
Then for any $\lambda >0$ we have
\begin{align*}
\bar f(t)
&= g_\lambda(t) + b_{\lambda}(t),
\end{align*}
where
\begin{align*}
g_\lambda (t) = 
\begin{cases}
& \bar f(t) \qquad \text{if}~t \in \left( \cup_{j}Q_j \right)^c \\
&\frac{1}{|Q_j|} \int_{Q_j}\bar f(r)dr \qquad  \text{if}~ t \in  Q_j,
\end{cases}
\end{align*}
$$
b_{\lambda}(t) := \bar f(t) -g_\lambda (t)=\sum_{j}1_{Q_j}(t)\left( \bar f(t) - \int_{Q_j}\bar f(r)dr \right)
=: \sum_j b_{\lambda,j}(t),
$$
and $Q_j= Q_{k_j,l_j}$ for some $k_j,l_j \in \bZ$ . 
Here $g_\lambda$ and $b_\lambda$ satisfy the followings:
\begin{enumerate}[(i)]
\item $$
\|g_\lambda \|_{L_1(\fR)} \leq \|\bar f\|_{L_1(\fR)} \quad \text{and} \quad \|g_\lambda\|_{L^\infty} \leq 2 \lambda,
$$
\label{i 1017 1}
\item 
Each $b_j$ is supported in a dyadic cube $Q_j$. Furthermore, $Q_j$ and $Q_\ell$ are disjoint if $j \neq \ell$. 
\item 
$$
\int_{Q_j}b_{\lambda,j}(x) dx=0.
$$
\item
$$
\|b_{\lambda,j}\|_{L_1(\fR)} \leq 2^{2} \lambda |Q_j|.
$$
\item
$$
\sum_{j} |Q_j| \leq \lambda^{-1} \|\bar f\|_{L_1(\fR)}.
$$
\end{enumerate}
Denote $f_T (t,x) = 1_{0<t<T} f(t,x)$ and consider the decomposition
$$
f_T(t,x) = f_{\lambda,1}(t,x)+f_{\lambda,2}(t,x),
$$
where
\begin{align*}
f_{\lambda,1}(t,x)=
\begin{cases}
&f_T(t,x) \qquad \text{if}~t \in \left( \cup_{j}Q_j \right)^c \\
&\frac{1}{|Q_j|}\int_{Q_j} f_T(r,x)dr \qquad  \text{if}~ t \in  Q_j,
\end{cases}
\end{align*}
and
\begin{align*}
f_{\lambda,2}(t,x) = \sum_{j} 1_{Q_j}(t) \left( f_T(t,x) -  \frac{1}{|Q_j|}\int_{Q_j} f_T(r,x)dr\right).
\end{align*}
Note that $\|f_T(t,\cdot)\|_{ \dot \Lambda_m}= \bar f (t)$ and thus 
$$
\|f_{\lambda,1}(t,\cdot) \|_{  \dot \Lambda_m} \leq |g_\lambda(t)|.
$$
Hence by Theorem \ref{t 1018 1} and (i), there exists a positive constant $N_0(d,\gamma, \nu,m)$ so that
$$
\|\cG (f_{\lambda,1})\|_{L_\infty( (0,T) ; \dot \Lambda_{m+\gamma})}
\leq \frac{N_0}{2} \sup_{t} | g_\lambda (t)| \leq N_0\lambda.
$$
Therefore, 
\begin{align*}
& |\{ t \in (0,T) : \|\cG f(t,\cdot)\|_{ \dot \Lambda_{m+\gamma}} >  4N_0\lambda \}| \\
&\leq  |\{ t \in (0,T) : \|\cG f_{\lambda,1}(t,\cdot)\|_{ \dot \Lambda_{m+\gamma}} >  2N_0 \lambda \}|
+|\{ t \in (0,T) : \|\cG f_{\lambda,2}(t,\cdot)\|_{\dot \Lambda_{m+\gamma}} >  2N_0 \lambda \}| \\
&= |\{ t \in (0,T) : \|\cG f_{\lambda,2}(t,\cdot)\|_{ \dot \Lambda_{m+\gamma}} > 2N_0 \lambda \}|.
\end{align*}
We split the last term above into two parts.
By (v) and Chebyshev's inequality,
\begin{align*}
&|\{ t \in (0,T) : \|\cG f_{\lambda,2}(t,\cdot)\|_{ \dot \Lambda_{m+\gamma}} > 2 N_0 \lambda \}| \\
&\leq \sum_j |Q^\ast_j| 
+|\{ t \in (0,T) \cap \left( \cup_j Q^\ast_j\right)^c : \|\cG f_{\lambda,2}(t,\cdot)\|_{\dot \Lambda_{m+\gamma}} > 2N_0 \lambda \}| \\
&\leq 3 \lambda^{-1}\int_0^T\|f(t,\cdot)\|_{\Lambda_m}dt +
N \lambda^{-1} \int_{ (0,T) \cap \left( \cup_j Q_j^\ast\right)^c}  \|\cG f_{\lambda,2}(t,\cdot)\|_{\dot \Lambda_{m+\gamma}} dt,
\end{align*}
where $ \left( \cup_j Q_j^\ast \right)^c = \fR \setminus  \cup_j Q_j^\ast$.
It only remains to show that
$$
 \int_{ (0,T) \cap \left( \cup_j Q_j^\ast\right)^c}  \|\cG f_{\lambda,2}(t,\cdot)\|_{ \dot \Lambda_{m+\gamma}} dt
\leq N\int_0^T\|f(t,\cdot)\|_{ \dot \Lambda_m}dt.
$$
We put $Q_j=[t_j, t_j+\delta_j)$. 
Then $Q_j^\ast = [t_j-\delta_j, t_j+2\delta_j)$. 
Obviously, if $t \in (-\infty,  t_j - \delta_j)$ and $s \in Q_j$, then $1_{s<t} p(s,t,x) = 0$. 
Moreover, for any $t \in (0,T)$,
\begin{align*}
&\cG f_{\lambda,2}(t,x) \\
&= \int_0^t p(s,t,\cdot) \ast f_{\lambda,2}(s,\cdot)(x) ds  \\
&= \sum_{j} \int_{\fR} \left(1_{0<s<t<T} p(s,t,\cdot) -1_{t_j<t}p(t_j,t,\cdot)\right)  \\
& \qquad \qquad \qquad \qquad \qquad \qquad \qquad \ast 1_{Q_j}(s) \left( f_T(s,\cdot) -  \frac{1}{|Q_j|}\int_{Q_j} f_T(r,\cdot)dr\right) (x) ds
\end{align*}
since
\begin{align*}
\int_{\fR} 1_{Q_j}(s) \left( f_T(s,y) -  \frac{1}{|Q_j|}\int_{Q_j} f_T(r,y)dr\right)  ds =0 \qquad \forall y \in \fR^d.
\end{align*}
Therefore by Lemma \ref{l 1017 3}, 
\begin{align*}
&\int_{ (0,T) \cap \left( \cup_j Q_j^\ast\right)^c}  \|\cG f_{\lambda,2}(t,\cdot)\|_{\dot \Lambda_{m+\gamma}} dt \\
&\leq \sum_j \int^T_{(t_j+2\delta_j) \wedge T}  \|\cG f_{\lambda,2}(t,\cdot)\|_{\dot \Lambda_{m+\gamma}} dt \\
&\leq \sum_j \int_{(t_j+2\delta_j) \wedge T}^T \int_{ Q_j} \Bigg\|\left(p(s,t,\cdot) - p(t_j,t,\cdot) \right)  \\
&\qquad \qquad \qquad \qquad \qquad \qquad \qquad 
\ast\left( f_T(s,\cdot) -  \frac{1}{|Q_j|}\int_{Q_j} f_T(r,\cdot)dr\right)\Bigg\|_{\dot \Lambda_{m+\gamma}} ds dt\\
&\leq N\sum_j \int_{(t_j+2\delta_j) \wedge T}^T \bigg(\int_{ Q_j} \delta_j (t-s)^{-2}  \|f_T(s,\cdot)\|_{\dot \Lambda_m} ds \\
&\qquad \qquad \qquad \qquad \qquad \qquad+ \int_{ Q_j} \delta_j (t-s)^{-2}  ds \frac{1}{|Q_j|}\int_{Q_j} \|f_T(r,\cdot)\|dr \bigg)dt \\
&\leq N \int^T_{0}\|f(s,\cdot)\|_{\dot \Lambda_m} ds.
\end{align*}
The lemma is proved. 
\end{proof}
\begin{corollary}
						\label{c 1023 1}
Let $T>0$ and $p \in (1,\infty)$. 
Then there exist positive constant $N_1$ and $N_2$ so that
\begin{align}
					\label{e 17 0712 1}
\|\cG f\|_{L_p((0,T) ; \dot \Lambda_{\gamma+m})} \leq N_1 \|f\|_{L_p((0,T) ; \dot \Lambda_{m})} 
 \quad \forall f \in L_1\left((0,T); \Lambda_m\right) \cap L_\infty\left((0,T); \dot \Lambda_m\right),
\end{align}
and
\begin{align}
					\label{e 17 0712 2}
\|\cG f\|_{L_p((0,T) ; \Lambda_{\gamma+m})} 
\leq N_2 \|f\|_{L_p((0,T) ;   \Lambda_{m})} 
 \quad \forall f \in L_1\left((0,T); \Lambda_m\right) \cap L_\infty\left((0,T);  \Lambda_m\right),
\end{align}
where $N_1=N_1(d,p,\gamma, \nu,m)$ and $N_2=N_2(d,p,\gamma, \nu,m,T)$.
\begin{proof}
(\ref{e 17 0712 1}) is an easy application of the Marcinkiewicz interpolation theorem with Theorem \ref{t 1018 1} and Lemma \ref{l w11}.
(\ref{e 17 0712 2}) comes from  (\ref{l infty bounded}) and (\ref{e 17 0712 1}) 
\end{proof}
\end{corollary}
The following lemma will be used to show the existence of a solution to equation (\ref{main eqn}).
\begin{lemma}
					\label{l 1023 1}
For $\phi \in C_c^\infty\left((0,T) \times \fR^d\right)$, 
denote
$$
f_{\phi }(t,x) =  -\phi_t(t,x)  -\psi^\ast (t,i\Delta)\phi(t,x).
$$
Then for all $(s,y) \in (0,T) \times \fR^d$,
\begin{align*}
\int_0^T \int_{\fR^d} 1_{s<t}  p(s,t,x-y) (f_{\phi}(s,x)) dx dt =\phi(s,y).
\end{align*}
\end{lemma}
\begin{proof}
It suffices to show that
\begin{align}
						\label{sol re 2}
\cF\left[\int_s^T \int_{\fR^d}   p(s,t,x) (f_{\phi\ast}(s,x+\cdot)) dt dx \right](\xi)
=\cF[\phi(s,\cdot)](\xi).
\end{align}
Denote $\hat \phi(t,\xi) := \cF[ \phi(t,\cdot)](\xi)$ and recall the property of the Fourier transform that
$$
\cF\left[\int_{\fR^d} f(x) g(x+\cdot) dx \right](\xi)
= \cF[f](-\xi) g(\xi).
$$
Thus by (\ref{ad fr})  and the integration by parts, 
the left hand side of (\ref{sol re 2}) is equal to
\begin{align*}
&\int_s^T \exp\left(\int_s^t\psi(r,-\xi)dr \right) 
\left(-\frac{\partial \hat \phi}{\partial t}(t,\xi) - \psi (t,-\xi) \hat \phi(t,\xi)\right) dt \\
&= \hat \phi(s,\xi)
+\int_s^T \psi(t,-\xi)\exp\left(\int_s^t \psi (r,-\xi)dr \right)  \hat \phi(t,\xi)  dt\\ 
&\quad -\int_{t}^T\exp\left(\int_s^t\psi(r,-\xi)dr \right) \psi(t,-\xi) \hat \phi(t,\xi) dt
= \hat \phi(s,\xi).
\end{align*}
The lemma is proved. 
\end{proof}

\vspace{3mm}
{\bf The proof of Theorem \ref{main thm}}
\vspace{2mm}

{\bf Part I.} (Uniqueness)
\vspace{2mm}

Let $u \in L_p((0,T); \Lambda_{m+\gamma})$ be a solution to the equation
\begin{align*}
&\frac{\partial u}{\partial t}(t,x)=\psi(t,i\nabla)u(t,x),\quad (t,x) \in (0,T) \times \fR^d \\
& u(0,x)=0,
\end{align*}
Then by the definition of the solution,
for any $\phi \in C^\infty_c\left((0,T) \times \fR^d\right)$
\begin{align*}
\int_0^T \int_{\fR^d} u(t,x) \left( -\phi_t(t,x) - \psi^\ast (t,i\nabla)\phi(t,x)\right)dtdx=0.
\end{align*}
Since $u$ is a locally integrable function on $(0,T) \times \fR^d$, it suffices to show that
\begin{align}
						\label{e 1019 4}
\{\left( -\phi_t(t,x) -  \psi^\ast (t,i\nabla)\phi(t,x)\right) : \phi(t,x) \in C_c\left( (0,T) \times \fR^d  \right) \}
\end{align}
is dense in $L_p((0,T) \times \fR^d)$.  
However, thie is an application of an $L_p$-theory, which is proved in author's previous paper (\cite{kim2016q}).
Note that  in \cite{kim2016q}, we considered the equation:
\begin{align*}
&\frac{\partial u}{\partial t}(t,x)= \psi^\ast(t,i\nabla)u(t,x)-\lambda u (t,x)+g(t,x),\quad (t,x) \in (0,\infty) \times \fR^d,
~\lambda > 0 .
\end{align*}
However following \cite[Section 2.5]{Krylov2008}, we can obtain the unique solvability of the Cauchy problem (\ref{main eqn}) and considering the change of variable $t \rightarrow T-t$, we can also have the unique solvability of the terminal value problem.
Thus for any $g \in L_p((0,T) \times \fR^d)$, there exists a unique solution 
$v \in L_p((0,T); H_p^\gamma(\fR^d)) \cap H^1_p((0,T); L_p(\fR^d))$
to the equation 
\begin{align*}
&-\frac{\partial v}{\partial t}(t,x)=\psi^\ast(t,i\nabla)v(t,x)+g(t,x),\quad (t,x) \in (0,T) \times \fR^d \\
& v(T,x)=0,
\end{align*}
and 
$$
\|v_t\|_{L_p((0,T);H_p^\gamma(\fR^d))}+\|v\|_{L_p((0,T);H_p^\gamma(\fR^d))} \leq N \|g\|_{L_p((0,T) \times \fR^d)},
$$
where $N$ is independent of $g$, $H_p^\gamma(\fR^d)$ is the fractional Sobolev space with the order $\gamma$ and the exponent $p$, and $H^1_p((0,T); L_p(\fR^d))$ is the $L_p(\fR^d)$-valued Sobolev space with  the order $1$ and the exponent $p$. 
Since the terminal value is zero
and the operator $  \phi \to -\phi_t- \psi^\ast (t,i\nabla)\phi $ is a continuous operator from 
$L_p((0,T); H_p^\gamma(\fR^d)) \cap H^1_p((0,T); L_p(\fR^d))$ 
to $L_p((0,T) \times \fR^d)$, 
we can find a sequence of $v_n(t,x) \in C^\infty_c ((0,T) \times \fR^d)$ such that
$$
\| v - v_n \|_{H^1_p((0,T);H_p^\gamma(\fR^d))} \to 0
$$
and 
$$
\| g_n  - g\|_{L_p((0,T \times \fR^d))} \to 0
$$
as $n \to \infty$, 
where $g_n = -\frac{\partial v_n}{\partial t}(t,x)-\psi^\ast (t,i\nabla)v_n(t,x)$.
Therefore the set in (\ref{e 1019 4}) is dense in $L_p((0,T) \times \fR^d)$ and the uniqueness is proved.

\vspace{2mm}
{\bf Part II.} (Existence)
\vspace{2mm}

For a $f \in L_p \left((0,T) ; \Lambda_{m}\right)$, we claim that
$$
u(t,x):= \cG f(t,x) :=  \int_0^t p(s,t,\cdot) \ast f(s,\cdot)(x) ds
$$
is a solution to equation (\ref{main eqn}).
By Definition \ref{sol}, it is sufficient to show that for any $\phi \in C^\infty_c\left((0,T) \times \fR^d\right)$
\begin{align*}
\int_0^T \int_{\fR^d} u(t,x) \left( -\phi_t(t,x) - \psi^\ast (t,i\nabla)\phi(t,x)\right)dtdx
=\int_0^T \int_{\fR^d} f(t,x) \phi(t,x) dtdx.
\end{align*}
Recall the notation
$$
f_{ \phi }(t,x) :=  -\phi_t(t,x)  - \psi^\ast (t,i\Delta)\phi(t,x).
$$
By Fubini's theorem and Lemma \ref{l 1023 1},
\begin{align*}
&\int_0^T \int_{\fR^d} u(t,x) \left( -\phi_t(t,x) - \psi^\ast(t,i\nabla)\phi(t,x)\right)dtdx \\
&=\int_0^T \int_{\fR^d} \int_s^T  \int_{\fR^d} 1_{s<t} p(s,t,x-y)  f_{\phi}(t,x)dt dx f(s,y) dy ds \\
&=\int_0^T \int_{\fR^d} \phi(s,y) f(s,y) dy ds.
\end{align*}

\vspace{2mm}
{\bf Part III.} ($L_p\left((0,T); \Lambda_{m+\gamma}\right)$-estimate)
\vspace{2mm}

Due to Part I and II, $\cG f(t,x)$ is the unique solution to (\ref{main eqn}).
Therefore (\ref{main est}) holds due to  Corollary \ref{c 1023 1} if
$ f \in L_1\left((0,T); \Lambda_m\right) \cap L_\infty\left((0,T); \Lambda_m\right)$.
For general $f  \in L_p\left((0,T); \Lambda_m\right)$, it suffices to find an approximation 
$f_n  \in L_1\left((0,T); \Lambda_m\right) \cap L_\infty\left((0,T); \Lambda_m\right)$
so that
$$
\|f_n-f\|_{L_p\left((0,T); \Lambda_m\right)} \to 0
$$
as $n \to \infty$. 
However, one can easily find this approximation by mollifying and cutting off $f$ with respect to $t$. 
The theorem is proved. \qed 

\vspace{3mm}
{\bf The proof of Theorem \ref{main thm 2}}
\vspace{2mm}
											
Since the existence and uniqueness of a solution $u$ to equation (\ref{main eqn}) is already proved, it suffices to show
(\ref{main est 2}).
Recall that $u(t,x):=\cG f(t,x)$ is the solution to equation (\ref{main eqn}). 
Due to (\ref{e 1023 2}),
\begin{align}
						\label{e 1023 4}
\| u\|_{L_p\left((0,T); C^{n+\alpha}\right)}
\leq N\|f\|_{L_p\left((0,T); C^{n+\alpha} \right)}.
\end{align}

Note that $\Lambda_{\alpha} = C^{\alpha}$. 
Thus due to (\ref{main est}), for any multi-index $|\beta|=n$,

\begin{align}
						\label{e 1023 5}
\| D^{\beta} u\|_{L_p\left((0,T); \Lambda_{\gamma+\alpha}\right)}
\leq N\|D^{\beta}f\|_{L_p\left((0,T); C^\alpha \right)}.
\end{align}
Moreover, due to \cite[Corollary 6.3.10]{grafakos2009modern}, for any multi-index $\alpha_1 = \lfloor \gamma +\alpha \rfloor$, 
\begin{align}
						\notag
\| D^{\beta+\alpha_1} u\|_{L_p\left((0,T); C_{\gamma+\alpha-\lfloor \gamma +\alpha \rfloor}\right)} 
&=\| D^{\beta+\alpha_1} u\|_{L_p\left((0,T); \Lambda_{\gamma+\alpha-\lfloor \gamma +\alpha \rfloor}\right)} \\
						\label{e 1023 6}
&\leq N\| D^{\beta} u\|_{L_p\left((0,T); \Lambda_{\gamma+\alpha}\right)}
\end{align}
since $\gamma + \alpha \notin \bZ_+$,
Finally, combining (\ref{e 1023 4})-(\ref{e 1023 6}), we obtain (\ref{main est 2}).
The theorem is proved. \qed.

\end{document}